\documentclass[12pt]{amsart}
\usepackage[margin=1.2in]{geometry}

\usepackage{amsfonts}
\usepackage{amssymb}
\usepackage{amsmath}
\usepackage{amsthm}
\usepackage{mathtools}

\usepackage{tikz}
\usepackage{tikz-cd}
\usetikzlibrary{arrows,chains,matrix,positioning,scopes}

\usepackage{xcolor}

\makeatletter
\tikzset{join/.code=\tikzset{after node path={%
\ifx\tikzchainprevious\pgfutil@empty\else(\tikzchainprevious)%
edge[every join]#1(\tikzchaincurrent)\fi}}}
\makeatother
\tikzset{>=stealth',every on chain/.append style={join},
         every join/.style={->}}
\tikzstyle{labeled}=[execute at begin node=$\scriptstyle,
   execute at end node=$]

\theoremstyle{plain}
\newtheorem{theorem}{Theorem}[section]
\newtheorem{lemma}[theorem]{Lemma}
\newtheorem{proposition}[theorem]{Proposition}
\newtheorem{corollary}[theorem]{Corollary}

\theoremstyle{definition}

\newtheorem{remark}[theorem]{Remark}

\numberwithin{equation}{section}

\newcommand{\N}{\mathbb N}
\newcommand{\D}{\mathbb D}

\newcommand{\T}{\mathbb T}
\newcommand{\C}{\mathbb C}
\newcommand{\R}{\mathbb R}
\newcommand{\Z}{\mathbb Z}

\title{Some results about spaceability in function spaces}

\author[A. Debrouwere]{Andreas Debrouwere}
\address{A. Debrouwere, Department of Mathematics and Data Science \\ Vrije Universiteit Brussel, Belgium\\ Pleinlaan 2 \\ 1050 Brussels \\ Belgium}
\email{Andreas.Debrouwere@vub.be}



\begin{document}

\begin{abstract}
%In this article, we establish the spaceability of sets of nowhere regular  periodic functions, for various notions of regularity, including H\"older regularity and real analyticity. Our work solves several questions posed by  Bernal-Gonz\'alez et al. \cite{}.

We investigate two problems concerning spaceability in function spaces. First, we study the spaceability of sets of nowhere regular  periodic functions, for various notions of regularity, including H\"older regularity and real analyticity. Second, we show that the set of real analytic functions is spaceable in $C((0,1))$. In fact, we provide a complete characterization of the closed subspaces of $C((0,1))$ consisting of real analytic functions. Our work solves several open questions posed by  Bernal-Gonz\'alez et al. \cite{BG,BG0,BG1}.
\end{abstract}

\maketitle

\section{Introduction}

A recurrent phenomenon in analysis is that sets of functions exhibiting a certain irregular behaviour are large, both in the topological and algebraic sense. For example, the set of continuous nowhere differentiable functions on $[0,1]$ is comeager in $C([0,1])$ and contains an infinite-dimensional linear subspace, except for zero. The same holds for the set of smooth nowhere real analytic functions on $[0,1]$ within $C^\infty([0,1])$. We refer to \cite[Chapter 1]{linbook} for these and related results.

In this article, we are interested in the following notion of largeness: A subset $M$ of a topological vector space $E$ is said to be \emph{spaceable in $E$} if $M \cup \{0\}$ contains a closed infinite-dimensional subspace of $E$. This terminology was coined by Gurariy and Quarta \cite{GQ}. We study two problems  related to spaceability in function spaces.

\subsection{Spaceability of sets of nowhere regular periodic functions} A classical result of Fonf, Gurariy, and Kadets \cite{FGK} says that the set of continuous nowhere differentiable functions on $[0,1]$ is spaceable in $C([0,1])$; see \cite{Bayart,Berezhnoi,Hencl,RP} for improvements and related results. Here, we study the spaceability of various  sets of nowhere regular periodic functions. We consider H\"older regularity and ultradifferentiable classes, including real analyticity,  and work with spaces of periodic functions whose spectrum is contained in a prescribed set $\Gamma \subseteq \Z$, which includes spaces of disk algebra type ($\Gamma = \N$)  as an  important particular instance.

In the H\"older setting, we answer several questions posed by Bernal-Gonz\'alez et al.\ in \cite[p.\ 15]{BG0} and  \cite[p.\ 600]{BG1} about the spaceability of sets of functions in the disk algebra and holomorphic H\"older spaces  on the unit disk $\D$ whose restriction to  $ \T = \partial \D$ is nowhere regular in a certain sense. In the real analytic case, we show that the set of functions in the smooth disk algebra $A^\infty(\D)$ whose restriction to $\T$ is nowhere analytic is spaceable in $A^\infty(\D)$. This is related to the question on  \cite[p.\ 607]{BG1}, asking whether the smaller set of functions in $A^\infty(\D)$ whose restriction to $\T$ has a Pringsheim singularity at every point of $\T$ is spaceable in $A^\infty(\D)$.  Note that our result implies that the set of smooth nowhere real analytic functions on $[0,1]$ is spaceable in $C^\infty([0,1])$. For  general ultradifferentiable classes, our work complements the genericity results from \cite{Esser}, where  spaceability is not considered. Moreover, unlike \cite{Esser}, we also cover the quasianalytic case.

Our approach combines properties of lacunary Fourier series \cite{Katznelson,Zygmund}, expressing the principle that such series have the same regularity everywhere, together with a new general spaceability result for sets of periodic functions.
% defined by support restrictions and decay conditions on their Fourier coefficients.

\subsection{Spaceability  properties of the set of real analytic functions} Gurariy \cite{Gurariy1966} proved that the set of differentiable functions is not spaceable in $C([0,1])$; see the introduction of \cite{RP} for an overview of related results. In contrast, Bernal-Gonz\'alez \cite{BG} showed that the set of smooth functions is spaceable in $C((0,1))$, and asked  \cite[p.\ 1294, Question 3]{BG} whether the same holds for the smaller set  of real analytic functions. We give a positive answer to this question and, in fact, obtain a complete characterization of the closed subspaces of $C((0,1))$ consisting of real analytic functions. Furthermore, we show that every complemented subspace of $C((0,1))$ consisting of real analytic functions  is finite-dimensional.

Doma\'nski and Langenbruch \cite{DL}  characterized the Fr\'echet subspaces of the space of real analytic functions on an open set $\Omega \subseteq \R^d$ by using composition operators and abstract functional analysis. We rely heavily on this characterization for $\Omega = (0,1)$ (and its proof).  Bernal-Gonz\'alez's original approach is different and uses the theory of M\"untz spaces \cite{GL}. We also indicate how the spaceability of the set  of real analytic functions in $C((0,1))$ can be obtained via this way.

\section{Function spaces} \label{sect-FS}
In this preliminary section, we introduce the function spaces that will be used throughout the article.

\subsection{ Wiener-Beurling-type spaces} \label{subsect-WB}

We denote by $L^1_{2\pi}$ the Banach space  (of equivalence classes) of $2\pi$-periodic measurable functions $f: \R \to \C$ such that %Given $f \in L^1(\T)$, we write $\tilde{f}(t) = f(e^{it})$, $t \in \R$. We endow $L^1(\T)$ with the norm
$$
\|f\|_{L^1_{2\pi}} = \frac{1}{2\pi} \int_{-\pi}^{\pi} |f(t)| {\rm d} t < \infty.
$$

We define the \emph{Fourier coefficients} of an element $f \in L^1_{2\pi}$ as
$$
\widehat{f}(n) = \frac{1}{2\pi}  \int_{-\pi}^{\pi}  f(t) e^{-int} {\rm d} t, \qquad n \in \Z. 
$$

Let $E$ be a topological vector space that is continuously included in $L^1_{2\pi}$. For $\Gamma \subseteq \Z$, we write $E_\Gamma$ for the closed subspace of $E$ consisting of all $f \in E$ whose spectrum is contained in $\Gamma$, i.e., $\{ n\in \Z \mid \widehat{f}(n) \neq 0 \} \subseteq \Gamma$.

Let $a = (a_n)_{n \in  \Z}$ be a sequence of positive numbers with $a_n \to 0$ as $|n| \to \infty$. We denote by $\mathcal{F}\ell^\infty_a$ the space consisting of all $f \in L^1_{2\pi}$ such that
$$
|f|_{\mathcal{F}\ell^\infty_a} = \sup_{n \in \Z} \frac{|\widehat{f}(n)|}{a_n} < \infty.
$$
We endow $\mathcal{F}\ell^\infty_a$ with the norm
$$
\|f\|_{\mathcal{F}\ell^\infty_a} = \|f\|_{L^1_{2\pi}} + |f|_{\mathcal{F}\ell^\infty_a}, \qquad f \in \mathcal{F}\ell^\infty_a.
$$
Then, $\mathcal{F}\ell^\infty_a$ is a Banach space that is continuously included in $L^1_{2\pi}$.  
We define $\mathcal{F}c_{0,a}$ as the space consisting of all $f \in L^1_{2\pi}$ such that $\widehat{f}(n) = o(a_n)$ as $|n| \to \infty$. Note that $\mathcal{F}c_{0,a}$ is a closed subspace of  $\mathcal{F}\ell^\infty_a$.

\subsection{H\"older and Lipschitz regularity} 
We denote by $C_{2\pi}$ the Banach space of $2\pi$-periodic  continuous functions on $\R$. %, endowed with the norm %Given $f \in L^1(\T)$, we write $\tilde{f}(t) = f(e^{it})$, $t \in \R$. We endow $L^1(\T)$ with the norm
%$$
%\|f\|_{C_{2\pi}} = \sup_{|t| \leq \pi} |f(t)|.
%$$

%For $S \subseteq \C$, we denote by $C(S)$ the Banach space of continuous complex-valued functions on $S$, endowed with the sup-norm:
%$$
%\| f \|_{C(S)} = \sup_{z \in S } |f(z)|, \qquad f \in C(S).
%$$
%Let $\D$ be the open unit disk in $\C$.The \emph{disk algebra} $A(\D)$ is the Banach space of all continuous functions on $\overline{\D}$ that are analytic on $\D$, endowed with the norm $\| \, \, \|_{C(\D)}$. The following result is well-known, see e.g.\ \cite[Chapter 20, Theorem 4.2]{Conway}.

%\begin{proposition} \label{prop-DA}
%The restriction map $A(\D) \to (C(\T))_+, \, f \mapsto f_{\mid \T}$ is an isometric isomorphism.
%\end{proposition}
Let $\alpha \in (0,1)$. The \emph{$\alpha$-H\"older space} $\Lambda^\alpha_{2\pi}$ consists of all $f \in C_{2\pi}$ such that 
$$
|f|_{\Lambda^\alpha_{2\pi}} = \sup_{\substack{t \in [-\pi,\pi] \\ 0< |h| \leq 1}} \frac{|f(t+h) - f(t)|}{|h|^\alpha} < \infty. 
$$
We endow $\Lambda^\alpha_{2\pi}$ with the norm
$$
\|f\|_{\Lambda^\alpha_{2\pi}} = |f(0)| + |f|_{\Lambda^\alpha_{2\pi}}, \qquad f \in \Lambda^\alpha_{2\pi}.
$$
Then, $\Lambda^\alpha_{2\pi}$ is a Banach space that is continuously included in  $C_{2\pi}$. The \emph{little $\alpha$-H\"older space} $ \lambda^\alpha_{2\pi}$ consists of all $f \in C_{2\pi}$ such that  $f(t+h) - f(t) = o(|h|^\alpha)$ as $|h| \to 0$ uniformly for $t \in [-\pi,\pi]$. Note that $\lambda^\alpha_{2\pi}$ is a closed subspace of  $\Lambda^\alpha_{2\pi}$.
% For $0 < \alpha < \beta < 1$ it holds that $\Lambda^\beta_{2\pi} \subseteq \lambda^\alpha_{2\pi}$ with continuous inclusion.

%We will be mostly concerned with the spaces $C^\alpha(\T)$ and $c^\alpha(\T)$. 

%Every element of $f \in C^\alpha(\D)$ is uniformly continuous on $\D$ and therefore has a unique continuous extension to $\overline{\D}$. Therefore, we may view $C^\alpha(\D)$ as a subspace of $A(\D)$. In view of Proposition \ref{prop-DA}, the following result follows from  \cite[Theorem 1.2.7]{Sewell} and \cite[Theorem 41 and the remark at the end of Section 5]{HL}).
%\begin{proposition} \label{prop-CA}
%Let $\alpha \in (0,1)$. 
%\begin{itemize}
%\item[(i)] The restriction map $C^\alpha(\D) \to (C^\alpha(\T))_+, \, f \mapsto f_{\mid \T}$ is an isometric isomorphism.
%\item[(ii)] The restriction map $c^\alpha(\D) \to (c^\alpha(\T))_+, \, f \mapsto f_{\mid \T}$ is an isometric isomorphism.
%\end{itemize}
%\end{proposition}

Next, we introduce pointwise H\"older and Lipschitz conditions. Let $\alpha \in (0,1)$ and $t_0 \in \R$. A function $f: \R \to \C$ is said to belong to $\Lambda^\alpha(t_0)$ if $f(t_0+h) - f(t_0) = O(|h|^\alpha)$ as $h \to 0$, and to  $\lambda^\alpha(t_0)$ if  $f(t_0+h) - f(t_0) = o(|h|^\alpha)$ as $h \to 0$. Similarly, we write $f \in \operatorname{Lip}(t_0)$ if $f(t_0+h) - f(t_0) = O(|h|)$ as $h \to 0$.
%\begin{remark}
%Let $\alpha \in (0,1)$. Then, $f \in C^\alpha(\T)$ ($\tilde{f} \in c^\alpha(\T)$) if and only if $\tilde{f} \in C^\alpha(\R)$ ($\tilde{f} \in c^\alpha(\R)$). A similar remark holds pointwise  H\"older and Lipschitz conditions.

\subsection{Ultradifferentiable classes}
%Let $I \subseteq \R$ be an open interval. We write $C^\infty(I)$  for the Fr\'echet space of all smooth functions on $I$. 
We denote by $C^\infty_{2\pi}$ the Fr\'echet space of $2\pi$-periodic  smooth functions on $\R$. 
We have the following  characterization of  $C^\infty_{2\pi}$ in terms of Fourier coefficients:
\begin{proposition}\label{prop-FS} $\displaystyle C^\infty_{2\pi} = \bigcap_{k>0} \mathcal{F}\ell^\infty_{(|n|^{-k})}$.
\end{proposition}
%We denote by $A^\infty(\D)$ the space of all infinitely differentiable complex-valued functions on $\overline{\D}$ that are analytic on $\D$. For $n \in \N$ we set
%$$
%\| f \|_{n} = \max_{p \leq n }\sup_{z \in \D} |f^{(p)}(z)|, \qquad f \in A^\infty(\D).
%$$
%We endow $A^\infty(\D)$ with the Fr\'echet space topology generated by the system of norms $\{ \| \cdot \|_n \mid n \in \N\}$. 

%We write $C^\infty(\R)$  for the Fr\'echet space of all infinitely differentiable complex-valued functions on $\R$. %For $K \subseteq \R$ compact and $n \in \N$ we  set
%$$
%\| f \|_{K,n} = \max_{p \leq n} \sup_{t \in K} |f^{(p)}(t)|, \qquad f \in C^\infty(\R).
%$$
%We endow $C^\infty(\R)$ with the Fr\'echet space topology generated by the system of norms $\{ \| \cdot \|_n \mid n \in \N\}$. 
%We define $C^\infty(\T)$ as the space consisting of all $f \in C(\T)$ such that $\tilde{f} \in C^\infty(\R)$ and endow this space with the initial topology with respect to the mapping $C^\infty(\T) \to C^\infty(\R), \, f \mapsto \widetilde{f}$.

Next, we introduce ultradifferentiable classes defined via weight sequences; see \cite{Komatsu1} for more information. By a \emph{weight sequence}, we mean here a sequence $M = (M_p)_{p \in \N}$ of  positive numbers with $M^{1/p}_p \to \infty$ as $p \to \infty$ that satisfies the following two conditions:
\begin{itemize}
	\item[$(M.1)$]  $M$ is \emph{log-convex}:
	$$
	M^2_{p} \leq M_{p-1}M_{p+1}, \qquad \forall p \in \N \backslash \{0\}.
	$$
	\item[$(M.2)'$] $M$ is \emph{shift-stable}: There is $H >0$ such that
	$$M_{p+1} \leq H^{p+1} M_{p}, \qquad \forall p \in \N.$$
	%\item[$()$]   $M$ is \emph{non-quasianalytic}:
	%$$\displaystyle \sum_{p = 1}^\infty \frac{M_{p-1}}{M_p} < \infty;$$
	%\item[$(M.4)$]  The sequence $(M_p/p!)$ is log-convex.
	%  $\displaystyle \exists C >0 \, \forall p,q \in \N  : \frac{M_{p}}{p!}\frac{M_{q}}{q!} \ \leq C^{p+q} \frac{M_{p+q}}{(p+q)!}$.
\end{itemize}
Examples of weight sequences  are the Gevrey sequences $(p!^s)$ with $s >0$.

The \emph{associated function} $\omega_M$ of a weight sequence $M$ is defined as
$$
\omega_M(\rho) = \sup_{p \in \N} \log \frac{M_0\rho^p}{M_p}, \qquad \rho > 0,
$$
and $\omega_M(0) = 0$. Since $M^{1/p}_p \to \infty$ as $p \to \infty$, it holds that that $\omega_M(t) < \infty$ for all $t \geq 0$. Furthermore, $\log t = o(\omega_M(t))$ as $t \to \infty$.
%means that the sequence $M$ is log-convex, while $(M.3)'$ is known as the \emph{non-quasianalyticity condition}. Condition $%(M.4)$ is equivalent to the sequence $(M_p/p!)^{1/p}$ being almost increasing, that is, 
%$$
%\exists C > 0 \, \forall p \leq q : \left(\frac{M_p}{p!} \right)^{1/p} \leq C\left(\frac{M_q}{q!}\right)^{1/q}.
%$$
%In this regard, we mention that $(M.3)$ implies that  \cite[]{KomatsuI}
%$$
%\lim_{p \to \infty} \left ( \frac{M_p}{p!}\right)^{1/p} = \infty.
%$$
%We refer to \cite{Komatsu1}. Prime examples of weight sequences  satisfying all the above conditions are the sequences $p!^{\sigma}$ with $\sigma>1$.  

Let $M$ and $N$ be two weight sequences. We write $M \prec N$ if for every $h >0$ there is $C>0$ such that 
$$M_p \leq Ch^p N_p, \qquad \forall p \in \N. 
$$
%We use the notation $M \prec N$ to indicate that the latter inequality remains valid for all $H >0$ and a suitable $C = C_H >0$.

Let $M$ be a weight sequence. Given  a compact interval $J \subseteq \R$ and $h >0$,  we write $\mathcal{E}^{M,h}(J)$ for the Banach space consisting of all $f \in C^\infty(J)$ such that
$$
\| f \|_{\mathcal{E}^{M,h}(J)} = \sup_{p \in \N}\sup_{t \in J} \frac{|f^{(p)}(t)|}{h^pM_{p}} < \infty. 
$$
Let $I \subseteq \R$ be an open interval. We define the \emph{spaces of ultradifferentiable functions of Roumieu type $\{M\}$ and Beurling type $(M)$ on $I$} as
$$
\mathcal{E}^{\{M\}}(I) = \varprojlim_{J \Subset I} \varinjlim_{h \to \infty}\mathcal{E}^{M,h}(J), \qquad \mathcal{E}^{(M)}(I) =  \varprojlim_{J \Subset I} \varprojlim_{h \to 0^+}\mathcal{E}^{M,h}(J),
$$
where $J \Subset I$ means that $J$ is a compact subinterval of $I$.  Note that $\mathcal{E}^{\{p!\}}(I)$ is equal to the space $\mathcal{A}(I)$ of real analytic functions on $I$. We use $\mathcal{E}^{[M]}(I)$ as a common notation for $\mathcal{E}^{\{M\}}(I)$  and $\mathcal{E}^{(M)}(I)$; a similar convention will be used for other spaces as well.

Let $M$ and $N$ be two weight sequences. If $M \prec N$, then $\mathcal{E}^{\{M\}}(I) \subseteq \mathcal{E}^{(N)}(I)$ with continuous inclusion.

Let $M$ be a weight sequence. We denote by $\mathcal{E}^{[M]}_{2\pi}$ the closed subspace of $\mathcal{E}^{[M]}(\R)$ consisting of $2\pi$-periodic functions. We define
$$
\mathcal{F}\ell^\infty_{\{\omega_M\}} = \bigcup_{h>0} \mathcal{F}\ell^\infty_{(e^{-\omega_M(|n|/h)})}, \qquad \mathcal{F}\ell^\infty_{(\omega_M)} = \bigcap_{h>0} \mathcal{F}\ell^\infty_{(e^{-\omega_M(|n|/h)})}.
$$
We have the following characterization of  $\mathcal{E}^{[M]}_{2\pi}$ in terms of Fourier coefficients (see e.g.\ \cite[Proposition 2]{Debrouwere}\footnote{In \cite{Debrouwere}, the stronger condition (M.2) \cite{Komatsu1} instead of $(M.2)'$ is assumed. However, an inspection of the proof of \cite[Proposition 2]{Debrouwere} shows that $(M.2)'$ in fact suffices for this result.}): %or \cite[Satz 3.8]{Petzsche}.  
\begin{proposition}\label{FCultra}
Let $M$ be a weight sequence. Then, $\mathcal{E}^{[M]}_{2\pi} = \mathcal{F}\ell^\infty_{[\omega_M]}$. 
%$$
%\mathcal{E}^{\{M\}}_{2\pi} =   \bigcup_{h <0} \mathcal{F}\ell^\infty_{(e^{-\omega(|n|/h)})}, \qquad \mathcal{E}^{(M)}_{2\pi} =   \bigcap_{h <0} \mathcal{F}\ell^\infty_{(e^{-\omega(|n|/h)})}.
%$$
\end{proposition}

\subsection{The disk algebra and related spaces}

We denote by $\D$  the open unit disk in $\C$ and by $\T$ its boundary. For  $f: \T \to \C$ we define the $2\pi$-periodic function $\widetilde{f}(t) = f(e^{it})$, $t \in \R$.

The \emph{disk algebra} $A(\D)$ is the Banach space of all continuous functions on $\overline{\D}$ that are analytic on $\D$. We have the following characterization of $A(\D)$ in terms of Fourier coefficients:
\begin{proposition}\label{disk-1}
The map $A(\D) \to C_{2\pi,\N}, \, f \mapsto \widetilde{f_{\mid \T}}$, is an isometric isomorphism. 
\end{proposition}

 Let $\alpha \in (0,1)$.  The \emph{holomorphic $\alpha$-H\"older space} $\Lambda^\alpha(\D)$  consists of all $f \in A(\D)$ such that 
$$
|f|_{\Lambda^\alpha(\D)} = \sup_{\substack{z,w \in \D \\ z \neq w}} \frac{|f(z) - f(w)|}{|z-w|^\alpha} < \infty. 
$$
We endow $\Lambda^\alpha(\D)$ with the norm
$$
\|f\|_{\Lambda^\alpha(\D)} = |f(0)| + |f|_{\Lambda^\alpha(\D)}, \qquad f \in \Lambda^\alpha(\D).
$$
Then, $\Lambda^\alpha(\D)$ is a Banach space that is continuously included in $A(\D)$. The \emph{holomorphic little $\alpha$-H\"older space} $\lambda^\alpha(\D)$ consists of all $f \in A(\D)$ such that  $f(z) - f(w) = o(|z-w|^\alpha)$ as $|z-w| \to 0$ uniformly for $z,w \in \D$. Note that $\lambda^\alpha(\D)$ is a closed subspace of  $\Lambda^\alpha(\D)$. We have the following characterization of $\Lambda^\alpha(\D)$ and $\lambda^\alpha(\D)$  in terms of Fourier coefficients:

\begin{proposition} \label{disk-2}  Let $\alpha \in (0,1)$. 
\begin{itemize}
 \item[(i)]  \cite[Theorem 1.2.7]{Sewell} The map $\Lambda^\alpha(\D) \to \Lambda^\alpha_{2\pi,\N}, \, f \mapsto \widetilde{f_{\mid \T}}$, is a topological isomorphism.
 \item[(ii)] \cite[Theorem 41 and the remark at the end of Section 5]{HL}  The map $\lambda^\alpha(\D) \to \lambda^\alpha_{2\pi,\N}, \, f \mapsto \widetilde{f_{\mid \T}}$, is a topological isomorphism.
\end{itemize}
\end{proposition}

Next, we discuss pointwise H\"older and Lipschitz conditions. Let $\alpha \in (0,1)$ and $z_0 \in \T$. A function $f: \T \to \C$ is said to belong to $\Lambda^\alpha(z_0)$ if $f(z) - f(z_0) = O(|z-z_0|^\alpha)$ as $z \to z_0$, $z \in \T$. The spaces $\lambda^\alpha(z_0)$ and $\operatorname{Lip}(z_0)$ are defined similarly.
\begin{remark}\label{rem-point}
Given $\alpha \in (0,1)$ and $t_0 \in \R$, it is clear that a function  $f: \T \to \C$ belongs to $\Lambda^\alpha(e^{it_0})$ if and only if $\widetilde{f} \in \Lambda^\alpha(t_0)$. A similar result holds for the pointwise $\lambda^\alpha$- and $\operatorname{Lip}$-spaces. 
\end{remark}
The \emph{smooth disk algebra} $A^\infty(\D)$ is the Fr\'echet space of all smooth functions on $\overline{\D}$ that are analytic on $\D$.  We have the following characterization of $A^\infty(\D)$ in terms of Fourier coefficients:
\begin{proposition}\label{disk-3}
The map $A^\infty(\D) \to C^\infty_{2\pi,\N}, \, f \mapsto \widetilde{f_{\mid \T}}$, is a topological isomorphism.
\end{proposition}

\section{Spaceability of sets of nowhere regular periodic functions}
In this section, we  establish the spaceability of various sets of nowhere regular  periodic functions. We consider H\"older and Lipschitz regularity and ultradifferentiable classes, including real analyticity. 
\subsection{A spaceability criterion about Beurling-Wiener-type spaces} We start with showing a general spaceability result involving the spaces $\mathcal{F}\ell^\infty_{a}$ and $\mathcal{F}c_{0,a}$.
\begin{proposition}\label{main} Let $a = (a_n)_{n \in  \Z}$ be a sequence of positive numbers with $a_n \to 0$ as $|n| \to \infty$ and $\Gamma \subseteq \Z$ be infinite. Let $E$ be a topolgical vector space such that $E$ is continuously included in $L^1_{2\pi}$ and $  \mathcal{F}\ell^\infty_{a,\Gamma} \subseteq E_{\Gamma}$. Then, the set $E_\Gamma \backslash  \mathcal{F}c_{0,a}$ is spaceable in $E_\Gamma$.
\end{proposition}
\begin{proof}
Choose sequences $\gamma_k = (\gamma_{k,j})_{j \in \N} \subseteq \Gamma$, $k \in \N$, such that 
$$
\{\gamma_{k,j} \mid j \in \N\} \cap \{\gamma_{k',j} \mid j \in \N\} =\emptyset, \qquad k\neq k',
$$
and 
$$
\sum_{j = 0}^\infty a_{\gamma_{k,j}} < \infty.
$$
Define
$$
f_k(t) = \sum_{j = 0}^\infty a_{\gamma_{k,j}} e^{i\gamma_{k,j}t}, \qquad t \in \R.
$$
Then, $f_k \in C_{2\pi}$ and, for $n \in \Z$,
$$
\widehat{f}_k(n)  = \left\{
 \begin{array}{ll}
    a_{\gamma_{k,j}}, &\mbox{if } n = \gamma_{k,j} \mbox{ for some } j \in \N,\\
   0, &\mbox{otherwise}.
  \end{array}
\right.
$$
We obtain that $f_k \in  \mathcal{F}\ell^\infty_{a,\Gamma} \subseteq E$. Define 
$$
X = \overline{\operatorname{span} \{f_k \mid k \in \N\}}^{E}.
$$
Note that $X$ is infinite-dimensional as the family $\{f_k \mid k \in \N\}$ is linearly independent. Let $f \in X$. Then,
$$
f = \lim_{N \to \infty} \sum_{m=0}^{\nu_N} \alpha_{m,N} f_m \, \, \mbox{ in }E,
$$
for certain $\nu_N \in \N$ and $\alpha_{m,N} \in \C$. As $E$ is continuously included in $L^1_{2\pi}$, we find that, for $n \in \Z$,
$$
\widehat{f}(n)  =  \lim_{N \to \infty} \sum_{m=0}^{\nu_N} \alpha_{m,N} \widehat{f}_m(n) = \left\{
 \begin{array}{ll}
   \alpha_k a_{\gamma_{k,j}}, & \mbox{if } n = \gamma_{k,j}  \mbox{ for some } k,j \in \N,\\
   0, &\mbox{otherwise},
  \end{array}
  \right.
$$
where $\alpha_k = \lim_{N \to \infty} \alpha_{k,N} \in \C$. Hence, $f \in E_\Gamma$. If $f \neq 0$, there exists $k_0 \in \N$ such that $\alpha_{k_0} \neq 0$ (otherwise, $\widehat{f}(n) =0$ for all $n \in \Z$ and thus $f= 0$). We obtain that, for all $j \in \N$,
$$
\frac{|\widehat{f}(\gamma_{k_0,j})|}{a_{\gamma_{k_0,j}}} = |   \alpha_{k_0}|. 
$$
Consequently, $f \notin \mathcal{F}c_{0,a}$. Thus, $X$ is an infinite-dimensional closed subspace of $E_\Gamma$ with $X \backslash \{0\} \subseteq E_\Gamma \backslash  \mathcal{F}c_{0,a}$. This shows the result.
\end{proof}

\subsection{Lacunary Fourier series and regularity}

A set $\Gamma \subseteq \Z$ is called \emph{(Hadamard) lacunary} if $\Gamma = (-\Gamma_1) \cup \Gamma_2$ with $\Gamma_k \subseteq \N$, $k= 1,2$, either finite or $\Gamma_k = \{ \gamma_{k,j} \mid  j \in \N\}$ such that
$$
\inf_{j \in \N} \frac{\gamma_{k,j+1}}{\gamma_{k,j}} > 1.
$$
As a rule of thumb, for periodic functions $f$ whose spectrum is contained in a lacunary set, the regularity of $f$ in the neigborhood of any point determines its global regularity. For H\"older and Lipschitz regularity, this may be expressed as follows (cf.\ \cite[Chapter V]{Katznelson}, \cite[Chapter 5, \textsection 6]{Zygmund}):
\begin{proposition} \label{lac-Holder} Let $\Gamma \subseteq \Z$ be lacunary.
\begin{itemize}
\item[(i)] Let $\alpha \in (0,1)$.  For $f \in L^1_{2\pi,\Gamma}$ the following statements are equivalent:
\begin{itemize}
\item[(a)] $f \in \Lambda^\alpha_{2\pi}$ ($f \in  \lambda^\alpha_{2\pi}$).
\item[(b)]  There is $t_0 \in \R$ such that $f \in \Lambda^\alpha(t_0)$ ($f \in \lambda^\alpha(t_0)$).
\item[(c)]  $f \in  \mathcal{F}\ell^\infty_{(|n|^{-\alpha})}$ ( $f \in \mathcal{F}c_{0,(|n|^{-\alpha})}$).
\end{itemize}
\item[(ii)] Let  $t_0 \in \R$. If $f \in \operatorname{Lip}(t_0)$, then $f \in  \mathcal{F}\ell^\infty_{(|n|^{-1})}$.
%Moreover,  $(C^\alpha(\T))_{\Gamma} =   (\mathcal{F}\ell^\infty_{(|n|^{-\alpha})})_{\Gamma}$ as TVS.
\end{itemize}
\end{proposition}

Next, we consider ultradifferentiability. We believe the following result is known, but we give a proof for the sake of completeness. 
\begin{proposition}\label{lacultra}
 Let $M$ be a weight sequence, $\Gamma \subseteq \Z$ be lacunary, and $f \in C_{2\pi,\Gamma}$. If $f_{\mid I}\in \mathcal{E}^{[M]}(I)$ for some non-empty open interval $I \subseteq \R$, then $f \in \mathcal{E}^{[M]}_{2\pi}$.
\end{proposition}
\begin{proof}
We may assume that $I \subseteq (-\pi,\pi)$. Fix a non-empty compact subinterval $J$ of $I$. Since $\Gamma$ is lacunary, \cite[p.\ 203, Lemma 6.5]{Zygmund} %and a simple rescaling argument
 implies that there is $C >0$ such that for all $f \in C_{2\pi,\Gamma}$ and $n \in \Z$
$$
 |\widehat{f}(n)| \leq C \left(\int_{J} |f(t)|^2 {\rm d} t \right)^{1/2}.
$$
Let $h >0$ and $f \in C_{2\pi,\Gamma}$ be such that $f_{\mid I} \in \mathcal{E}^{M,h}(I)$. Then, for all $p \in \N$ and $n \in \Z$,
$$
|n|^p|\widehat{f}(n)| = | \widehat{f^{(p)}}(n)| \leq C  \left(\int_{J} |f^{(p)}(t)|^2 {\rm d} t \right)^{1/2} \leq C|J|^{1/2} \| f \|_{\mathcal{E}^{M,h}(J)} h^p M^p.
$$
Hence,
$$
|\widehat{f}(n)| \leq  C|J|^{1/2} \| f \|_{\mathcal{E}^{M,h}(J)} \inf_{p \in \N} \frac{ h^p M_p}{|n|^p} = M_0C|J|^{1/2} \| f \|_{\mathcal{E}^{M,h}(J)} e^{-\omega_M(|n|/h)}.
$$
The result now follows from Proposition \ref{FCultra}.
\end{proof}

In the remainder of this section, we combine Proposition \ref{main} with the above two results to obtain several spaceability results about nowhere regular periodic functions.

\subsection{H\"older and Lipschitz regularity} We start with a result about continuous nowhere H\"older regular periodic functions.
%\begin{theorem} Let $\alpha \in (0,1)$ and $\Gamma \subseteq \Z$ be infinite. The set 
%\begin{equation}
%\label{set1}
%\{f \in \Lambda^\alpha_{2\pi, \Gamma} \mid f  \notin  \lambda^\alpha(t) \mbox{ for all } t \in \R\}
%\end{equation}
%is spaceable in $\Lambda^\alpha_{2\pi, \Gamma}$. 
%(and thus also in $\Lambda^\alpha_{2\pi}$).
%\end{theorem}
%\begin{proof}
%We may assume that $\Gamma$ is lacunary (otherwise, replace $\Gamma$ by a lacunary subset of $\Gamma$).  Proposition \ref{lac-Holder} yields that $\Lambda^\alpha_{2\pi, \Gamma}  =   \mathcal{F}\ell^\infty_{(|n|^{-\alpha}),\Gamma}$. By Proposition \ref{main}, $\Lambda^\alpha_{2\pi, \Gamma}  \backslash  \mathcal{F}c_{0,(|n|^{-\alpha})}$ is spaceable in $\Lambda^\alpha_{2\pi, \Gamma} $. Another application of Proposition \ref{lac-Holder} gives that the latter set coincides with the one in \eqref{set1}.
%\end{proof}

\begin{theorem} \label{thm2} Let $\Gamma \subseteq \Z$ be infinite. The set 
\begin{equation}
\label{set2}
\{f \in C_{2\pi,\Gamma} \mid f \notin  \Lambda^\alpha(t) \mbox{ for all } t \in \R \mbox{ and } \alpha \in (0,1)\}
\end{equation}
is spaceable in $C_{2\pi, \Gamma}$. % (and thus also in $\lambda^\alpha_{2\pi}$).
\end{theorem}
\begin{proof}
Define $a = (1/\log(e+|n|))$. We may assume that $\Gamma$ is lacunary and that 
$\sum_{\gamma \in \Gamma} a_\gamma < \infty$  (otherwise, replace $\Gamma$ by a  subset of $\Gamma$ that satisfies these properties). Then, $\mathcal{F}\ell^\infty_{a,\Gamma} \subseteq C_{2\pi,\Gamma}$. Since $|n|^{-\alpha}= o(a_n)$ as $|n| \to \infty$ for all $\alpha \in (0,1)$, we obtain that
$$
\bigcup_{\alpha \in (0,1)} \mathcal{F}\ell^\infty_{(|n|^{-\alpha})} \subseteq \mathcal{F}c_{0,a}.
$$
By Proposition \ref{main}, $C_{2\pi,\Gamma} \backslash  \mathcal{F}c_{0,a}$ is spaceable in $C_{2\pi,\Gamma}$. Hence, also the bigger set
$$
C_{2\pi,\Gamma} \backslash \bigcup_{\alpha \in (0,1)}  \mathcal{F}\ell^\infty_{(|n|^{-\alpha})} 
$$
 is spaceable in $C_{2\pi,\Gamma}$. Proposition \ref{lac-Holder} gives that the latter set coincides with the one in \eqref{set2}.
\end{proof}

The following result gives a positive answer to the question posed on \cite[p.\ 600]{BG1}.

\begin{corollary} The set 
$$
\{f \in A(\D) \mid f_{\mid \T} \notin  \Lambda^\alpha(z) \mbox{ for all } z \in \T \mbox{ and } \alpha \in (0,1)\}
$$
is spaceable in $A(\D)$.
\end{corollary}
\begin{proof}
This follows from Proposition \ref{disk-1}, Remark \ref{rem-point}, and Theorem \ref{thm2} with $\Gamma = \N$.
\end{proof}

\begin{theorem} \label{thm3} Let $\alpha \in (0,1)$ and $\Gamma \subseteq \Z$ be infinite. The set 
$$
\{f \in \lambda^\alpha_{2\pi,\Gamma} \mid f \notin  \Lambda^\beta(t) \mbox{ for all } t \in \R \mbox{ and } \beta \in (\alpha,1)\}
$$
is spaceable in $ \lambda^\alpha_{2\pi, \Gamma}$.
% (and thus also in $\lambda^\alpha_{2\pi}$).
\end{theorem}
\begin{proof}
We may assume that  $\Gamma$ is lacunary. Proposition \ref{lac-Holder} implies that $\lambda^\alpha_{2\pi,\Gamma} =  \mathcal{F}c_{0,(|n|^{-\alpha}),\Gamma}$. Define $a = (|n|^{-\alpha}/\log(e+|n|))$.  Since $a_n = o(|n|^{-\alpha})$ and $|n|^{-\beta}= o(a_n)$ as $|n| \to \infty$ for all $\beta \in (\alpha,1)$, we obtain that
$$
\bigcup_{\beta \in (\alpha,1)} \mathcal{F}\ell^\infty_{(|n|^{-\beta})} \subseteq \mathcal{F}c_{0,a} \subseteq \mathcal{F}\ell^\infty_{a} \subseteq  \mathcal{F}c_{0,(|n|^{-\alpha})}.
$$
The remainder of the proof is similar to the one of  Theorem \ref{thm2}.
%By Proposition \ref{main}, $\lambda^\alpha_{2\pi,\Gamma} \backslash  \mathcal{F}c_{0,a}$ is spaceable in $\lambda^\alpha_{2\pi}$. Hence, also the bigger set
%$$
%\lambda^\alpha_{2\pi,\Gamma} \backslash \bigcup_{\alpha < \beta < 1} \mathcal{F}\ell^\infty_{(|n|^{-\beta})} 
%$$
% is spaceable in $\lambda^\alpha_{2\pi}$. Another application of Proposition \ref{lac-Holder} gives that the latter %set coincides with the one from \eqref{set1}.
\end{proof}
\begin{corollary} \label{cor-2}  Let $\alpha \in (0,1)$. The set
$$
\{f \in \lambda^\alpha(\D) \mid  f_{\mid \T} \notin  \Lambda^\alpha(z) \mbox{ for all } z \in \T \mbox{ and }  \beta \in (\alpha,1) \}
$$
is spaceable in $ \lambda^\alpha(\D)$.
\end{corollary}
\begin{proof}
This follows from Proposition \ref{disk-2}, Remark \ref{rem-point}, and Theorem \ref{thm3} with $\Gamma = \N$.
\end{proof}
\begin{remark}
Let $\alpha \in (0,1)$. Corollary \ref{cor-2} implies that the set
$$
\{f \in \lambda^\alpha(\D) \mid  f_{\mid \T} \notin  \operatorname{Lip}(z) \mbox{ for all } z \in \T \}
$$
is spaceable in $ \lambda^\alpha(\D)$. This gives a positive answer to the first part of question (a) posed on \cite[p.\ 15]{BG0}.
\end{remark}

%For $S \subseteq \C$, we denote by $C(S)$ the Banach space of continuous complex-valued functions on $S$, endowed with the sup-norm:
%$$
%\| f \|_{C(S)} = \sup_{z \in S } |f(z)|, \qquad f \in C(S).
%$$

%\begin{remark}\label{hollip}
%Let $\alpha \in (0,1)$. We define $\Lambda^\alpha(\D)$  as the space of all $f \in A(\D)$ such that 
%$$
%|f|_{\Lambda^\alpha(\D)} = \sup_{\substack{z,w \in \D \\ z \neq w}} \frac{|f(z) - f(w)|}{|z-w|^\alpha} < \infty. 
%%$$
%We endow $\Lambda^\alpha(\D)$ with the norm
%$$
%\|f\|_{\Lambda^\alpha(\D)} = |f(0)| + |f|_{\Lambda^\alpha(\D)}, \qquad f \in \Lambda^\alpha(\D).
%$$
%Then, $\Lambda^\alpha(\D)$ is a Banach space. We write $\lambda^\alpha(\D)$ for the space of all $f \in A(\D)$ such that  $f(z) - f(w) = o(|z-w|^\alpha)$ as $|z-w| \to 0$ uniformly for $z,w \in \D$. Note that $\lambda^\alpha(\D)$ is a closed subspace of  $\Lambda^\alpha(\D)$.

%The linear map $f \mapsto \tilde{f}$ induces topological isomorphisms $\Lambda^\alpha(\D) \to \Lambda^\alpha_{2\pi,\N}$  \cite[Theorem 1.2.7]{Sewell} and  $\lambda^\alpha(\D) \to \lambda^\alpha_{2\pi,\N}$  \cite[Theorem 41 and the remark at the end of Section 5]{HL}.  Hence, Theorem \ref{thm2} with $\Lambda = \N$ implies that the set 
%$$
%\{f \in \lambda^\alpha(\D) \mid  f_{\mid \T} \notin  \Lambda^\alpha(z) \mbox{ for all } z \in \T,  \beta \in (\alpha,1) \}
%$$
%is spaceable in $ \lambda^\alpha(\D)$. Consequently, also the bigger set
%%
%\end{remark}

For $\alpha \in (0,1]$ we define the Fr\'echet space
$$
\widetilde{\Lambda}^\alpha_{2\pi} = \varprojlim_{\beta \to \alpha-} \Lambda^\beta_{2\pi}.
$$
\begin{theorem}\label{thm3bis}  Let $\Gamma \subseteq \Z$ be infinite.
\begin{itemize}
\item[(i)] Let $\alpha \in (0,1)$. The set
$$
\{f \in \widetilde{\Lambda}^\alpha_{2\pi,\Gamma}  \mid f \notin  \Lambda^\alpha(t) \mbox{ for all } t \in \R\}
$$
is spaceable in $\widetilde{\Lambda}^\alpha_{2\pi, \Gamma}$. %(and thus also in $\widetilde{\Lambda}^\alpha_{2\pi}$).
\item[(ii)] The set 
$$
\{f \in \widetilde{\Lambda}^1_{2\pi,\Gamma} \mid f \notin  \operatorname{Lip}(t) \mbox{ for all } t \in \R\}
$$
is spaceable in $\widetilde{\Lambda}^1_{2\pi, \Gamma}$. %(and thus also in $\widetilde{\Lambda}^1_{2\pi}$).
\end{itemize}
\end{theorem}
\begin{proof}
 Let $\alpha \in (0,1]$ (we prove (i) and (ii) simultaneously). We may assume that  $\Gamma$ is lacunary.  Proposition \ref{lac-Holder} yields that 
$$ \widetilde{\Lambda}^\alpha_{2\pi,\Gamma}  = \bigcap_{\beta \in (0,\alpha)}\mathcal{F}\ell^\infty_{(|n|^{-\beta}),\Gamma}.
$$ 
Define $a = (|n|^{-\alpha}\log(e+|n|))$. Since $a_n = O(|n|^{-\beta})$ and $|n|^{-\alpha}= o(a_n)$ as $|n| \to \infty$ for all $\beta \in (0,\alpha)$, we obtain that
$$
 \mathcal{F}\ell^\infty_{(|n|^{-\alpha})} \subseteq \mathcal{F}c_{0,a} \subseteq \mathcal{F}\ell^\infty_{a} \subseteq \bigcap_{\beta \in (0,\alpha)} \mathcal{F}\ell^\infty_{(|n|^{-\beta})}.
$$
The remainder of the proof is similar to the one of  Theorem \ref{thm2}.
\end{proof}

For $\alpha \in (0,1]$ we define the Fr\'echet space
$$
\widetilde{\Lambda}^\alpha(\D) = \varprojlim_{\beta \to \alpha-} \Lambda^\beta(\D).
$$
\begin{corollary}\label{cor3}  Let $\Gamma \subseteq \Z$ be infinite.
\begin{itemize}
\item[(i)] Let $\alpha \in (0,1)$. The set
$$
\{f \in \widetilde{\Lambda}^\alpha(\D)  \mid f_{\mid \T} \notin  \Lambda^\alpha(z) \mbox{ for all } z \in \T\}
$$
is spaceable in $\widetilde{\Lambda}^\alpha(\D)$. %(and thus also in $\widetilde{\Lambda}^\alpha_{2\pi}$).
\item[(ii)] The set 
\begin{equation}
\label{set3}
\{f \in \widetilde{\Lambda}^1(\D) \mid f_{\mid \T} \notin  \operatorname{Lip}(z) \mbox{ for all } z \in \T\}
\end{equation}
is spaceable in $\widetilde{\Lambda}^1(\D)$. %(and thus also in $\widetilde{\Lambda}^1_{2\pi}$).
\end{itemize}
\end{corollary}
\begin{proof}
This follows from Proposition \ref{disk-2}, Remark \ref{rem-point}, and Theorem \ref{thm3bis} with $\Gamma = \N$.
\end{proof}

\begin{remark}
The second part of Corollary \ref{cor3} addresses question (b) on \cite[p.\ 15]{BG0}, which asked what lineability properties the set in \eqref{set3} enjoys.
\end{remark}
\subsection{Ultradifferentiable classes} We start with the following technical lemma.
\begin{lemma}\label{ws}
Let $M$ be a weight sequence. For every $L >1$ it holds that $\omega_M(L\rho) - \omega_M(\rho) \to \infty$ as $\rho \to \infty$.
\end{lemma}
\begin{proof}
Define  $m_p = M_p/M_{p-1}$, $p \geq 1$, and the counting function $m(\lambda) = \sum_{m_p \leq \lambda} 1$, $\lambda \geq 0$. By \cite[Equation (3.11)]{Komatsu1}, we have the following representation of $\omega_M$:
$$
\omega_M(\rho) = \int_0^\rho \frac{m(\lambda)}{\lambda} {\rm d} \lambda, \qquad \rho \geq 0.
$$
For $L >1$ we obtain that
$$
\omega_M(L\rho) - \omega_M(\rho) = \int_{\rho}^{L\rho} \frac{m(\lambda)}{\lambda} {\rm d} \lambda \geq m(\rho)\log L,
$$
from which the result follows.
\end{proof}
Given a weight function $M$, a function $f: \R \to \C$ is said to be \emph{nowhere of class $\mathcal{E}^{[M]}$} if $ f_{\mid I} \notin  \mathcal{E}^{[M]}(I)$ for all non-empty open intervals $I \subseteq \R$. %Thus, $f$ is nowhere real analytic if $f$ is nowhere of class $\mathcal{E}^{\{p!\}}$.
\begin{theorem}\label{thmu1}  Let $M$ be a weight sequence and $\Gamma \subseteq \Z$ be infinite. The set 
\begin{equation}
\label{set4}
\{f \in C^\infty_{2\pi,\Gamma} \mid f \mbox{ is nowhere of class } \mathcal{E}^{\{M\}}\}
\end{equation}
is spaceable in $C^\infty_{2\pi,\Gamma}$. % (and thus also in $C^\infty_{2\pi}$).
\end{theorem}
\begin{proof}
 We may assume that $\Gamma$ is lacunary. Proposition \ref{prop-FS} yields that 
 $$
 C^\infty_{2\pi} = \bigcap_{k >0}\mathcal{F}\ell^\infty_{(|n|^{-k})}.
 $$ 
 Define $a = (e^{-\omega_M\left(\sqrt{|n|}\right)})$. Note that $a_n = O(|n|^{-k})$  as $|n| \to \infty$ for all $k >0$ and, by Lemma \ref{ws}, $e^{-\omega_M(|n|/h)} = o(a_n)$   as $|n| \to \infty$ for all $h >0$. Consequently,
$$
 \mathcal{F}\ell^\infty_{\{\omega_M\}} \subseteq \mathcal{F}c_{0,a} \subseteq \mathcal{F}\ell^\infty_{a}  \subseteq  \bigcap_{k >0}\mathcal{F}\ell^\infty_{(|n|^{-k})}.
$$
By Proposition \ref{main}, $C^\infty_{2\pi,\Gamma} \backslash  \mathcal{F}c_{0,a}$ is spaceable in $C^\infty_{2\pi,\Gamma}$. Hence, also the bigger set $C^\infty_{2\pi,\Gamma} \backslash \mathcal{F}\ell^\infty_{\{\omega_M\}}$
is spaceable in $C^\infty_{2\pi,\Gamma}$. Propositions \ref{FCultra} and \ref{lacultra} imply  that the latter set coincides with the one in \eqref{set4}.
\end{proof}

\begin{corollary}\label{cor-smooth} The set 
$$
\{f \in A^\infty(\D) \mid \widetilde{f_{\mid \T}} \mbox{ is nowhere real analytic on } \R \}
$$
is spaceable in  $A^\infty(\D)$.
\end{corollary}
\begin{proof}
This follows from Proposition \ref{disk-3} and Theorem \ref{thmu1} with $M = (p!)$ and $\Gamma = \N$. 
\end{proof}

\begin{remark}
Corollary \ref{cor-smooth} sheds some light on the question on  \cite[p.\ 607]{BG1},  which asked whether the smaller set
$$
\{f \in A^\infty(\D) \mid \widetilde{f_{\mid \T}} \mbox{ is Pringsheim singular at every point  of }  \R \}
$$
is spaceable in $A^\infty(\D)$. 
\end{remark}

%\begin{theorem}  Let $M$ be a weight sequence and $\Gamma \subseteq \Z$ be infinite. The set 
%$$
%\{f \in   \mathcal{E}^{\{M\}}_{2\pi, \Gamma} \mid f \mbox{ is nowhere of class } \mathcal{E}^{(M)}\}
%$$
%is spaceable in $\mathcal{E}^{\{M\}}_{2\pi, \Gamma}$ (and thus also in $\mathcal{E}^{\{M\}}_{2\pi}$).
%\end{theorem}
%\begin{proof}
%We may assume that $\Gamma$ is lacunary. Define $a = (e^{-\omega_M(|n|)})$. Proposition \ref{FCultra} yields that 
% $$
% \mathcal{E}^{\{M\}}_{2\pi} = \bigcup_{h >0} \mathcal{F}\ell^\infty_{(e^{-\omega_M(|n|/h)})}.
% $$ 
%By Lemma \ref{ws}, $e^{-\omega_M(|n|/h)} = o(a_n)$   as $|n| \to \infty$ for all $h < 1$. Consequently,
%$$
%\bigcap_{h >0} \mathcal{F}\ell^\infty_{(e^{-\omega_M(|n|/h)})} \subseteq \mathcal{F}c_{0,a} \subseteq \mathcal{F}\ell^\infty_{a}  \subseteq  \bigcup_{h >0} \mathcal{F}\ell^\infty_{(e^{-\omega_M(|n|/h)})} .
%$$
%The remainder of the proof is similar to the one of  Theorem \ref{thmu1}.
%\end{proof}

\begin{theorem} \label{themu3}  Let $M$ and $N$ be two weight sequences with $M \prec N$ and $\Gamma \subseteq \Z$ be infinite. The set 
$$
\{f \in   \mathcal{E}^{(N)}_{2\pi, \Gamma} \mid f \mbox{ is nowhere of class }  \mathcal{E}^{\{M\}}\}
$$
is spaceable in $\mathcal{E}^{(N)}_{2\pi, \Gamma}$.
\end{theorem}
\begin{proof}
We may assume that $\Gamma$ is lacunary. Define $Q = (\sqrt{M_pN_p})$. Then, $Q$ is a weight sequence satisfying $M \prec Q \prec N$.  By \cite[Lemma 3.10]{Komatsu1}, for every $h > 0$ there are $C_1,C_2 >0$ such that
\begin{equation}
\label{wsmiddle}
\omega_Q(\rho) \leq \omega_M(\rho/h) +  C_1 \quad \mbox{and} \quad \omega_N(\rho/h) \leq \omega_Q(\rho) +  C_2, \qquad \forall \rho \geq 0. 
\end{equation}
 Proposition \ref{FCultra} yields that $\mathcal{E}^{(N)}_{2\pi}  =  \mathcal{F}\ell^\infty_{(\omega_N)}$.
 %$$
 %\mathcal{E}^{(N)}_{2\pi} = \bigcap_{h >0} \mathcal{F}\ell^\infty_{(e^{-\omega_N(|n|/h)})}.
 %$$ 
 Define $a = (e^{-\omega_Q(|n|)})$.  By Lemma \ref{ws} and \eqref{wsmiddle}, $e^{-\omega_M(|n|/h)} = o(a_n)$ and  $a_n =O(e^{-\omega_N(|n|/h)})$ as $|n| \to \infty$ for all $h >0$. Consequently,
$$
\mathcal{F}\ell^\infty_{\{\omega_M\}} \subseteq \mathcal{F}c_{0,a} \subseteq \mathcal{F}\ell^\infty_{a}  \subseteq    \mathcal{F}\ell^\infty_{(\omega_N)}.
$$
The remainder of the proof is similar to the one of  Theorem \ref{thmu1}.
\end{proof}

\begin{remark}
 Let $M$ and $N$ be two weight sequences with $M \prec N$.  Since $\mathcal{E}^{(N)}_{2\pi}$ is a topological subspace of $\mathcal{E}^{(N)}(\R)$, Theorem \ref{themu3} (with $\Gamma = \Z$) implies that the set
\begin{equation}
\label{setrem}
\{f \in   \mathcal{E}^{(N)}(\R) \mid f \mbox{ is nowhere of class }  \mathcal{E}^{\{M\}}\}
\end{equation}
is spaceable in  $\mathcal{E}^{(N)}(\R)$. This complements \cite[Theorem 2.10]{Esser}, in which it is shown that the set  in \eqref{setrem} is $\mathfrak{c}$-dense-lineable  $\mathcal{E}^{(N)}(\R)$, under the assumption that $M$ is non-quasianalytic.
\end{remark}

\section{Spaceability properties of the set of real analytic functions} 
Let $I \subseteq \R$ be an open interval and write  $C(I)$ for the Fr\'echet space of continuous functions on $I$. This section is devoted to the study of the closed subspaces of $C(I)$ consisting of functions that are real analytic. In particular, we show that the set $\mathcal{A}(I)$ is spaceable in $C(I)$. As mentioned in the introduction, we rely heavily on the work \cite{DL} of  Doma\'nski and Langenbruch.

We denote by $\mathcal{O}(\D)$ the Fr\'echet space of analytic functions on $\D$. The following result is a direct consequence of the proof of \cite[Proposition 5.3]{DL}.

\begin{proposition}\label{compoper} 
Let $I \subseteq \R$ be a non-empty open interval. There exists a real analytic function $\varphi: I \to \D$ such that the composition operator
$$
\mathcal{O}(\D) \to C(I), \, f \mapsto f \circ \varphi, 
$$
is  a topological embedding. %Consequently, $C_\varphi(\mathcal{H}(\D))$ is a  infinite-dimensional closed subspace of $C(I)$ consisting of real analytic functions on $I$. In particular, $\mathcal{A}(I)$ is spaceable in $C(I)$.
\end{proposition}
%It suffices to consider the case $I = (4\pi, \infty)$, as there exists a real analytic diffeomorphism between any two non-empty open intervals of $\R$. Define
%$$
%\varphi: (4\pi, \infty) \to \D, \quad \varphi(t) = e^{it} \left ( 1 - \frac{1}{t} + \frac{2\pi}{t(t+2\pi)} \sin(t/2) \right).
%$$
%Clearly $\varphi$ is real analytic on $(4\pi, \infty)$. Note that
%$$
%|\varphi(t)| \leq 1 - \frac{1}{t+2\pi}, \qquad t > 4\pi.
%$$
%%Hence, for all  $x > 4\pi$,
%$$\sup_{t \in (4\pi, x]} |f \circ \varphi(t)| \leq \sup_{z \in \left(1- \frac{1}{x+2\pi}\right) \overline{\D}} |f(z)|, \qquad f \in \mathcal{O}(\D).
%$$
%This shows that $C_\varphi: \mathcal{O}(\D) \to C(I)$ is continuous. Set 
%$$
%r_t = \frac{1}{t} + \frac{2\pi}{t(t+2\pi)}, \qquad t > 4\pi. 
%$$
%Then,
%\begin{equation}
%\label{lowerbound}
%|\varphi(t)| \geq 1 - r_t, \qquad t > 4\pi.
%\end{equation}
%For all $k \in \N$, it holds that
%$$
%|\varphi(4\pi k)| < |\varphi(4\pi k+2\pi)| \qquad \mbox{and} \qquad |\varphi(4\pi k + \pi)| > |\varphi(4\pi k+3\pi)|. 
%$$
%Consequently, there is $t_k \in (4\pi k , 4 \pi k + \pi)$ such that $\gamma = \{ \varphi(t) \mid t \in [t_k, t_k  +2\pi]\}$ is a closed  simple curve. By \eqref{lowerbound} and the maximum modulus principle, we find that 
%$$
%\sup_{z \in (1-r_{4\pi k}) \overline{\D}} |f(z)| \leq \sup_{z \in \gamma} |f(z)| = \sup_{ t \in [t_k, t_k  +2\pi]} |f \circ \varphi (t)|, \qquad f \in \mathcal{O}(\D).
%$$
%This shows that $C_\varphi: \mathcal{O}(\D) \to C(I)$ is a topological embedding and  completes the proof.
Proposition \ref{compoper} implies the following result, which gives a positive answer to  \cite[p.\ 1294, Question 3]{BG}.
\begin{corollary}\label{cor-sa}
Let $I \subseteq \R$ be a non-empty open interval.  The set  $\mathcal{A}(I)$ is spaceable in $C(I)$.
\end{corollary}
In the next remark, we indicate an alternative proof of Corollary \ref{cor-sa} (with $I = (0,1)$) by means of the theory of M\"untz spaces \cite{GL}.  This approach was used by Bernal-Gonz\'alez in \cite[Theorem 4.4]{BG} to show that the set $C^\infty((0,1))$ is spaceable in $C((0,1))$.
\begin{remark}
Let $(\lambda_n)_{n \in \N}$ be an increasing sequence of positive numbers that satisfies $\inf_{n \in \N} \lambda_{n+1} - \lambda_n >0$ and $\sum_{n = 1}^\infty 1/\lambda_n < \infty$. Define
$$
E= \overline{ \operatorname{span} \{ t^{\lambda_n} \mid n \in \N\}}^{C((0,1))}.
$$
Clearly, $E$ is an infinite-dimensional closed subspace of $C((0,1))$. By using the same argument as in \cite[Theorem 6.2.3]{GL}, but using the coefficient bounds from \cite[Corollary 8.4.2]{GL} instead of \cite[Proposition 6.14 and Theorem 6.2.2]{GL},  we see that every function in $E$ is real analytic on $(0,1)$.
%We now show that $E$ in fact consists of real analytic functions on $(0,1)$. This provides an alternative proof of the fact that $\mathcal{A}((0,1))$ is spaceable in $C((0,1))$. Our argument is based on the following result for M\"untz polynomials \cite{}[]: For every $\varepsilon >0$ there are $C >0$ and a compact interval $J \subseteq (0,1)$ such that, for all $\nu \in \N$ and $\alpha_1, \ldots, \alpha_\nu \in \C$,
%\begin{equation}
%\label{inequ}
%|\alpha_n| \leq C(1+\varepsilon)^{\lambda_n} \sup_{t \in J} \left |\sum_{n = 1}^\nu \alpha_n t^{\lambda} \right|, \qquad n= 1, \ldots, \nu.
%\end{equation}
%Now let $f \in E$.  Then,
%$$
%f = \lim_{N \to \infty} \sum_{n=1}^{\nu_N} \alpha_{n,N} t^{\lambda_n} \quad \mbox{ in } C((0,1)),
%$$
%for certain $\nu_N \in \N$ and $\alpha_{n,N} \in \C$.  Inequality \eqref{inequ} implies that the limit $\alpha_n = \lim_{N \to \infty} \alpha_{n,N}$ exists for each $n \in \N$ and that, for every $\varepsilon >0$ there are $C >0$ and a compact interval $J \subseteq (0,1)$ such that
%$$
%|\alpha_n| \leq C(1+\varepsilon)^{\lambda_n} \| f \|_J, \qquad n \in \N.
%$$
%Hence, by \eqref{}, we find that
%$$
%f(t) = } \sum_{n=1}^{\infty} \alpha_{n} t^{\lambda_n}, \qquad t \in (0,1).
%$$
%Another application of \eqref{} gives that this series is convergent 
\end{remark}
%\begin{corollary}
%Let $I \subseteq \R$ be a non-empty open interval. Then, $\mathcal{A}(I)$ is spaceable in $C(I)$.
%\end{corollary
We  now give a characterization of the closed subspaces of $C(I)$ consisting of real analytic functions.
\begin{theorem}\label{char}
Let $I \subseteq \R$ be a non-empty open interval. A Fr\'echet space is isomorphic to a closed subspace of $C(I)$ consisting of real analytic functions on $I$ if and only if  it is isomorphic to  a closed subspace of $\mathcal{H}(\D)$.
%\begin{itemize}
%\item[(a)] $E$ is isomorphic to a closed subspace of $\mathcal{H}(\D)$.
%\item[(b)] $E$ is isomorphic to a closed subspace of $\mathcal{A}(I)$.
%\item[(c)] $E$ is isomorphic to a closed subspace of $C(I)$ consisting of real analytic functions on $I$.
%\end{itemize}
\end{theorem}
\begin{proof}
Proposition \ref{compoper} implies that any Fr\'echet space that is isomorphic to  a closed subspace of $\mathcal{H}(\D)$ is also  isomorphic to a closed subspace of $C(I)$ consisting of real analytic functions on $I$. Conversely, let $E$ be a Fr\'echet space that is  isomorphic to a closed subspace of $C(I)$ consisting of real analytic functions on $I$. In \cite[Proposition 5.2]{DL} it is shown that every closed Fr\'echet subspace of $\mathcal{A}(I)$ is isomorphic to a closed subspace of $\mathcal{H}(\D)$. Hence,   it suffices to show that  $E$ is isomorphic to a closed subspace of $\mathcal{A}(I)$. There exists a topological embedding $T: E \to C(I)$ with $T(E) \subseteq \mathcal{A}(I)$. Note that $\mathcal{A}(I)$ is webbed and $E$ is ultrabornological.  De Wilde's  closed graph theorem \cite[Theorem 24.31]{M-V} therefore implies that the map  $T: E \to \mathcal{A}(I)$ is continuous. Since $\mathcal{A}(I)$ is continuously included into $C(I)$,  we find that  $T: E \to \mathcal{A}(I)$  is a topological embedding, which yields the result. 
\end{proof}

\begin{remark}
 Note that $\mathcal{H}(\D) \cong \Lambda_0(j)$ via Taylor coefficients; see \cite[Chapter 29]{M-V} for the definition of power series spaces. We refer to  \cite[Satz 3.2]{Vogt} for an internal characterization of the closed subspaces of   $\Lambda_0(j)$. 
\end{remark}

In the next result, we show that $C(I)$ does not contain infinite-dimensional complemented subspaces consisting of real analytic functions:
\begin{proposition}\label{prop-comp}
Let $I \subseteq \R$ be a non-empty open interval. Every complemented subspace of $C(I)$ consisting of real analytic functions on $I$ is finite-dimensional.
\end{proposition}
\begin{proof}

Let $E$ be a complemented subspace of $C(I)$ consisting of real analytic functions on $I$. There exists a continuous linear projection $P: C(I) \to C(I)$ onto $E$. Fix a  compact interval $J_0 \subseteq I$ with non-empty interior. Choose a compact subinterval $J \subseteq I$ and $C >0$ such that
\begin{equation}
\label{propp}
\sup_{x \in J_0}|P(f)(x)| \leq C \sup_{x \in J} |f(x)|, \qquad \forall f \in C(I).
\end{equation}
%There are continuous linear maps $T: E \to C(I)$ and $S: C(I) \to E$ such that $T(E) \subseteq \mathcal{A}(I)$ and $S \circ T = \operatorname{id}_E$.  Fix a  compact interval $J_0 \subseteq I$ with non-empty interior. Choose a continuous seminorm $p$ on $E$ such that
%$$
%\| T(e) \|_{J_0} \leq p(e), \qquad x \in E.
%$$
%We claim that $p$ is a norm on $X$: Let $x \in X$ with $p(x) = 0$ and thus $\| T(x) \|_{J_0} = 0$. As $T(x) \in \mathcal{A}(I)$, the uniqueness property of real analytic functions imply that $T(x) = 0$. Since $T: X \to \C(I)$ is injective, we obtain that $x =0$.
%Next, pick $C >0$ and   a  compact interval $J \subseteq I$ such that
%\begin{equation}
%\label{propp}
%p(S(f)) \leq C\| f \|_{J}  , \qquad f \in C(I).
%\end{equation}
We write $\rho_J: C(I) \to C(J), \, f \mapsto f_{\mid J}$, for the restriction map. We claim that $\rho_{J \mid E}: E \to C(J)$ 
is a topological embedding. As $\rho_{J}(E) \subseteq \mathcal{A}(J) \subseteq C^1(J)$, the result would  follow from the fact  that the set of differentiable functions on $J$ is not spaceable in $C(J)$ \cite{Gurariy1966}.  We now show the claim. %In fact, we will prove the stronger statement that $\tilde{T}: E \to C(J)$ admits continuous linear left inverse. 
Since $P(C(I)) = E \subseteq \mathcal{A}(I)$, \eqref{propp} and the uniqueness property of real analytic functions imply that  $P(f) = 0$ for all $f \in C(I)$ with $f_{\mid J} = 0$. As the restriction $\rho_J: C(I) \to C(J)$ is a quotient map  and $P = 0$ on $\ker \rho_J$, there exists a continuous linear map $\tilde{P}: C(J) \to C(I)$ such that $\tilde{P} \circ \rho_J =  P$. Then, $\operatorname{id}_{E} = P_{\mid E} = \tilde{P} \circ \rho_{J \mid E}$, which implies the claim.
\end{proof} 
\begin{remark}
Proposition \ref{prop-comp} is related to the much deeper result \cite[Theorem 3.7]{DV} of Doma\'nski and Vogt that every complemented Fr\'echet subspace of $\mathcal{A}(I)$ is finite-dimensional. 
% In fact, there is an alternative proof of Proposition \ref{prop-comp} similar to the one  of \cite[Theorem 3.7]{DV} using the structure theory of Fr\'echet spaces \cite{M-V}. Let $E$ be a complemented subspace of $C(I)$ consisting of real analytic functions on $I$.  Then, $E$ is isomorphic to a closed subspace of $\mathcal{A}(I)$ (see the proof of Theorem \ref{char}). Hence, by  \cite[Theorem 3.7]{DV} , $E$ satisfies $(\underline{DN})$. Since $C(I)$ satisfies $(\overline{\overline{\Omega}})$ and this property is inherited by complemented subspaces,  $E$ also has this property. As $E$ satisfies both $(\underline{DN})$ and $(\overline{\overline{\Omega}})$, $E$ must be finite-dimensional (see the proof of \cite[Theorem 3.7]{DV}). %Our original proof of Proposition \ref{prop-comp}  is more elementary and, in our opinion, more insightful. 
\end{remark}

\begin{remark}
In contrast to Proposition \ref{prop-comp}, there do exist infinite-dimensional complemented subspaces of $C(I)$ consisting of smooth functions. We may suppose that $I = \R$. Pick $\varphi_n \in C^\infty(\R)$, $n \in \N$, with $\varphi_n(n) = 1$ and $\operatorname{supp} \varphi_n \subseteq [n- \frac{1}{2}, n + \frac{1}{2}]$. Define 
$$
E = \left \{ \sum_{n =0}^\infty \alpha_n \varphi_n \mid \alpha_n \in \C \mbox{ for all } n \in \N \right \}.
$$
Then, $E \subseteq C^\infty(\R)$ and $E$ is infinite-dimensional. The continuous linear map
$$
P: C(\R) \to C(\R), \, f \mapsto  \sum_{n =0}^\infty f(n) \varphi_n,
$$
is a projection onto $E$.

 This extends  \cite[Theorem 4.4]{BG}, in which it is shown that $C^\infty(I)$ is spaceable in $C(I)$. In addition, it gives an alternative proof of this  result.
\end{remark}

\end{document}